\documentclass[a4,11pt]{article}

\usepackage{paralist, amsmath, amsthm, amsfonts, amssymb, xy}
\usepackage[margin=1in]{geometry} 
\usepackage{hyperref}
\usepackage[T1]{fontenc}
\usepackage{calligra}

\input xy
\xyoption{all}

\newcommand{\NN}{\mathbb{N}}
\newcommand{\ZZ}{\mathbb{Z}}
\newcommand{\SM}{\mathbf{Sm}_k}
\newcommand{\SCH}{\mathbf{Sch}_k}

\newcommand{\srarrow}{\twoheadrightarrow}
\newcommand{\irarrow}{\hookrightarrow}

\newcommand{\flag}{\mathcal{F}\ell\,}
\newcommand{\Proj}{\mathbb{P}}
\newcommand{\Laz}{\mathbb{L}}
\newcommand{\trecd}{\cdot\cdot\cdot}
\newcommand{\tred}{\ldots}
\newcommand{\unddot}{_\textbf{\textbullet}}
\newcommand{\variables}{[x_1,\ldots,x_n,y_1,\ldots,y_n]}
\newcommand{\spec}{{\rm Spec\,}}

\newtheorem{theorem}{Theorem}[section]
\newtheorem{lemma}[theorem]{Lemma}
\newtheorem{proposition}[theorem]{Proposition}
\newtheorem{corollary}[theorem]{Corollary}
\newtheorem{definition}[theorem]{Definition}

\theoremstyle{definition} \newtheorem{remark}[theorem]{Remark}
\theoremstyle{definition} 
\theoremstyle{definition} 
\theoremstyle{definition} 
\date{}

\title{\textbf{A Thom-Porteous formula for connective $K$-theory using algebraic cobordism}}
\author{THOMAS  HUDSON}

\begin{document}

%
%
\maketitle
\begin{abstract}
We prove a formula for the push-forward class of Bott-Samelson resolutions in the algebraic cobordism ring of the flag bundle. We specialise our formula to connective K-theory providing a geometric interpretation to the double $\beta$-polynomials of Fomin and Kirillov by computing the fundamental classes of Schubert varieties. As a corollary we obtain a Thom-Porteous formula generalising those of the Chow ring and of the Grothendieck ring of vector bundles.

\vspace*{1\baselineskip}
\noindent\textit{Key Words:} Algebraic cobordism, Schubert varieties, Flag bundles.
 
\noindent\textit{Mathematics Subject Classification 2000:} Primary: 14C25, 14M15; Secondary: 14F05, 19E15.
\end{abstract}


\tableofcontents

\section{Introduction} 
The main purpose of this paper is to extend the known Thom-Porteous formulas to the more general context of oriented cohomology theories. Given a map of vector bundles $h:E\rightarrow F$ over a Cohen-Macaulay scheme $X$, the Thom-Porteous formula expresses the fundamental classes of the degeneracy loci
$$D_r(h)=\{x\in X\ |\ {\rm rank}(h(x):E(x)\rightarrow F(x))\leq r\}\ , \ 0\leq r\leq\text{min}(\text{rank} E, \text{rank} F)$$
as polynomials in the Chern classes of the two bundles, provided $\text{codim}(D_r(h),X)$ is the expected one. The formula owes its name to R. Thom, who in the context of topology conjectured the existence of a universal family of polynomials describing the cohomology classes of these loci, and I. R. Porteous who in \cite{SimplePorteous} identified the correct family, hence proving the formula. For what concerns the algebro-geometric setting, the first proof of the statement is due to Kempf-Laksov \cite{DeterminantalKempf}. In \cite{FlagsFulton} Fulton considered a more general family of degeneracy loci, constructed out of morphisms of bundles with full flags, and proved that the Chow ring fundamental classes of these loci are described by the double Schubert polynomials of Lascoux-Sch\"{u}tzenberger. The key idea in Fulton's approach is to reduce the problem to the universal case represented by $\Omega_\omega$, the Schubert varieties of the full flag bundle $\flag V$, and to recursively compute their classes by means of certain operators on $CH^*(\flag V)$. Even though they are not explicitly mentioned in Fulton's proof, this procedure can be reinterpreted in geometric terms by bringing into the picture a family of desingularizations $R_I\stackrel{r_I}\rightarrow \Omega_\omega$ known as Bott-Samelson resolutions: double Schubert polynomials naturally describe the push-forwards $r_{I*}[R_I]_{CH}$ and, since $r_I$ is a birational morphism, these coincide with the desidered classes $[\Omega_\omega]_{CH}$. 

The essentially functorial nature of the proof allowed Fulton and Lascoux, without any major modification, to obtain an analogue of this result for the Grothendieck ring of vector bundles: in \cite{PieriFulton} they proved that as elements of $K^0(\flag V)$, the structure sheaves of Schubert varieties can be described by means of the double Grothendieck polynomials. The circle was finally closed by Buch who in \cite{GrothendieckBuch} established the Thom-Porteous formula for $K^0$.
 
Given this state of things, it seems quite natural to ask if such a formula also holds for other functors which satisfy the same properties enjoyed by $CH^*$ and $K^0$ or, if not, up to which extent Fulton's proof can be generalized and what aspects are the source of problems. The family of functors that we intend to consider is that of oriented cohomology theories and in particular, among them, algebraic cobordism. In \cite{AlgebraicLevine} Levine and Morel introduced the notion of oriented cohomology theory on the category of smooth schemes over a field $k$ by extending Quillen's original definition for differentiable manifolds. In this framework they constructed $\Omega^*$, an oriented cohomology theory which, again in analogy with Quillen's results for complex cobordism, they prove to be universal if the field has characteristic 0: for any other theory $A^*$ there exists a unique morphism of oriented cohomology theories $\vartheta_A:\Omega^*\rightarrow A^*$. This in particular says that formulas that are obtained for $\Omega^*$ specialise to all other theories: a Thom-Porteous formula for algebraic cobordism would have to generalise the ones that are already known.

In order to understand in which ways a general oriented theory $A^*$ differs from the Chow ring, it can be helpful to consider the behaviour of the first Chern class with respect to tensor product: while $c^{CH}_1$ behaves linearly, this is not necessarily true for $c_1^A$. Instead, there exists a power series $F_A\in A^*( k)[[u,v]]$ such that for any two line bundles $L$ and $M$ over a smooth scheme $X$ one has 
$$c_1^A(L\otimes M)=F_A(c_1^A(L),c_1^A(M))\ .$$
Actually the pair $(A^*(k),F_A)$ constitutes a commutative formal group law of rank 1 and in particular this implies the existence of another power series $\chi_A\in A^*(k)[u]$, known as the formal inverse, which expresses the relation between the first Chern class of a bundle and that of its dual. As we will see these two power series play a major role in expressing the cobordism analogue of double Schubert polynomials.

 We now proceed to describe in detail our results. Our first goal  consists in the computation of the classes $\mathcal{R}_I:=r_{I*}[R_I]_\Omega$ as elements of $\Omega^*(\flag V)$. It is worth noticing that for the flag manifold (i.e. when $X=\spec k$) this computation has been indipendently performed by Calm\'es-Petrov-Zainoulline in \cite{SchubertCalmes} and by Hornbostel-Kiritchenko in \cite{SchubertHornbostel} and to some extent one can view our work as an extension of the latter. In fact we explicitly compute the fundamental class of the smallest Schubert variety $\Omega_{\omega_0}$ and, by making use of the operators
$$ \overline{A_i}:\Omega^*(\flag V)\rightarrow\Omega^*(\flag V) \quad, \quad
\overline{A_i}(f):=(1+\sigma_i)\frac{f}{F_\Omega(x_i,\chi_\Omega (x_{i+1}))}$$ 
considered by Hornbostel and Kiritchenko, we obtain the following 

\begin{theorem}
Let $V\rightarrow X$ be a vector bundle of rank $n$  with a fixed full flag of subbundles $V_1\subset \trecd \subset V_n=V$. Denote by $\pi^*V=Q_n\srarrow \trecd \srarrow Q_1$ the universal full flag of quotient bundles on $\flag V\stackrel{\pi}\rightarrow X$. Let $I=(i_1,\tred,i_l)$ be an $l$-tuple with $i_j\in\{1,\tred, n-1\}$ and let $R_I\stackrel{r_I}\rightarrow \flag V$ be the corresponding Bott-Samelson resolution. Then in $\Omega^*(\flag V)$ one has
\begin{align}
\mathcal{R}_I=A_{i_l}\trecd A_{i_1}[\Omega_{\omega_0}]_\Omega\ \text{ with }\  [\Omega_{\omega_0}]_\Omega=\prod_{i+j\leq n} F_\Omega(x_i,\chi_\Omega(y_j))\ , \label{eq main theorem}
\end{align}
where we set $x_i:=c_1(\textit{Ker}(Q_i\srarrow Q_{i-1}))$ and $y_i:=c_1(V_i/V_{i-1})$. 
\end{theorem}

Once this result has been established, it would be desirable to bring the fundamental classes $[\Omega_\omega]_\Omega$ into the picture, but one faces two obstructions: first of all in algebraic cobordism not all Schubert varieties come equipped with a fundamental class and moreover, as already pointed out by Hornbostel and Kiritchenko in the case of the flag manifold, the classes $\mathcal{R}_I$ associated to the same Schubert variety do not necessarily coincide. The second of these issues is essentially due to the fact that the operators $\overline{A_i}$, unlike their counterparts for $CH^*$ and $K^0$, do not satisfy the braid relations. One possible way out was suggested to us by a result of Bressler and Evens, who in \cite{BraidBressler} showed that a family of operators in the shape of $\overline{A_i}$ satisfy the braid relations if and only if the formal group law $F$ appearing in the denominator can be written as $F(u,v)=u+v-buv$, i.e. if it is multiplicative. For this reason we chose to restrict our attention to connective $K$-theory (denoted $CK^*$), an oriented cohomology theory that is universal among those with a multiplicative formal group law. This choice also allows us to disregard the first of the obstructions we mentioned:  in \cite{ConnectiveDai} Dai and Levine showed that $CK^*$ has a suitable notion of fundamental class for all  equi-dimensional schemes. 

As it was pointed out to us by Buch, the polynomials that are obtained by specialising the right hand side of (\ref{eq main theorem}) to connective $K$-theory are the double $\beta$-polynomials $\mathfrak{H}^{(\beta)}_\omega$. This family of polynomials, introduced by Fomin and Kirillov in \cite{GrothendieckFomin}, can be thought of as a unification of double Schubert and double Grothendieck polynomials and their definition was inspired by combinatorial considerations. Our approach, on the other hand, allows us to provide a geometric interpretation to these polynomials by relating them to the fundamental classes of Schubert varieties. 

\begin{theorem}
Under the hypothesis  of the preceding theorem one has
$$[\Omega_\omega]_{CK^*}=\mathfrak{H}^{(\beta)}_\omega(\textbf{x}_i,\textbf{y}_j)\ ,\ \omega\in S_n\ ,$$
as elements of $CK^*(\flag(V))$.
\end{theorem}

As immediate corollaries of this result one obtains a formula for the generalised degeneracy loci considered by Fulton and in particular the Thom-Porteous formula for connective $K$-theory. In fact,  the $\beta$-polynomials describing the loci $D_r(h)$ turn out to be symmetric in the Chern roots of the two bundles and therefore can be expressed in terms of their Chern classes. If we  denote the resulting polynomials by $\mathfrak{D}^{CK}_{(e,f,r)}$ with $e,f,r\in \NN$, then we have the following 

\begin{corollary}[Thom-Porteous formula]

Let $E\stackrel{h}\rightarrow F$ be a morphism of vector bundles of rank $e$ and $f$ over $X\in\SM$ and fix $r$ such that $0\leq r\leq min(e,f)$. Denote by $\textbf{t}$ the triple $(e,f,r)$ and assume that $codim(D_r(h),X)=(e-r)(f-r)$. Then in $CK^*(X)$  one  has
$$[D_r(h)]_{CK}=\mathfrak{D}^{CK}_{\textbf{t}}(c_i(F),c_j(E^\vee))\ .$$
\end{corollary}

Let us finish by describing the internal organisation of the paper. In the first section we review notations and results concerning algebraic cobordism and its relations with other oriented cohomology theories. In the second section we describe the geometric aspects of the problem, together with an outline of the proof of the classical results for $CH^*$ and $K^0$. Finally in the third section we apply Fulton's approach to algebraic cobodism first and later to connective $K$-theory, obtaining our main results. In the appendix we summarize the original definition of double $\beta$-polynomials and we derive a property necessary for the proof of the Thom-Porteous formula. 
     


\paragraph{Acknowledgements:} This work mainly consists of the results of my PhD thesis and I would like to take the opportunity to thank my advisors Marc Levine and Jerzy Weyman for suggesting this topic to me and for their uninterrupted help throughtout my studies and beyond. I also would like to thank Anders S. Buch and Jens Hornbostel for their careful reading of an early version of this work, for their useful suggestions and for their encouragement.

 The support of the NSF via the grant DMS-0801220 ''Motivic homotopy theory", of the Humbolt Foundation via the Humbold Professorship of Professor Levine and of the National Research Foundation of Korea (NRF) via the grant  funded by the Korean government (MSIP)  (No. 2013-042157) is gratefully acknowledged.

\paragraph*{Notations and conventions:}
 
Given a field $k$ of characteristic 0, we will denote by $\SCH$ the category of separated schemes of finite type over $\spec k$ and by $\SCH'$ its subcategory obtained by considering only projective morphisms. $\SM$ will represent the full subcategory of $\SCH$ consisting of schemes smooth and quasi-projective over $\spec k$. In general by smooth morphism we will always mean smooth and quasi-projective.

\section{Algebraic cobordism and other oriented cohomology theories}

The purpose of this section is to recall the definition of algebraic cobordism and of some related concepts as those of oriented cohomology theory, oriented Borel-Moore homology and formal group laws,  setting up the notations necessary to specialise cobordism formulas to other theories. For a detailed treatment of these topics we refer the reader to \cite{AlgebraicLevine}. We will also present some computation involving Chern roots and Chern classes.

\paragraph*{Oriented cohomology and oriented Borel-Moore homology theories:}

 Roughly speaking, an oriented cohomology theory (which we will abbreviate as OCT) is a contravariant functor from $\SM$ to the category of graded rings, together with a family of push-forward maps associated to projective morphisms. Such a functor has to satisfy, along with some obvious compatibilities,  two geometric properties: the extended homotopy property and the projective bundle formula.

 Most of the results in \cite{AlgebraicLevine} are actually obtained by making use of the dual (and for smooth schemes categorically equivalent) notion of oriented Borel-Moore homology theory (or OBM), which allows to enlarge the family of schemes taken under consideration. For our purposes an OBM will be a covariant functor from $\SCH'$ to the category graded abelian groups, endowed with an external product and pull-back maps for l.c.i. morphisms. As before, together with some functorial requirements, the theory is supposed to satisfy the extended homotopy property, the projective bundle formula,  as well as a property related to cellular decomposition.

Before we consider some examples, let us introduce the notion of fundamental class. Given an OBM $A_*$ it is possible to associate to any l.c.i. scheme $X$ its fundamental class by setting   
$$[X]_{A_*}:=\pi_X^*(1)\ ,$$
  where $\pi_X$ is the structural morphism and $1$ represents the identity in the coefficient ring. It is important to notice that this assignment is compatible with l.c.i. pull-backs and that when one restricts to smooth schemes and  considers the associated  OCT $A^*$, the fundamental class $[X]_{A^*}$ will coincide with the identity of $A^*(X)$.

Two fundamental examples of oriented Borel-Moore homology theories on $\SCH$ are given by the Chow group functor $X\mapsto CH_*(X)$ and by $G_0[\beta,\beta^{-1}]$, a graded version of the Grothendieck group of coherent sheaves $X\mapsto G_0(X)$. The graded structure is obtained by tensoring $G_0(X)$ with $\ZZ[\beta,\beta^{-1}]$ where $\text{deg}\, \beta=1$ and by modifying the pull-back and  push-forward maps as follows:
  $$f^*([\mathcal{E}]\cdot\beta^n)=[f^*(\mathcal{E})]\cdot \beta^{n+d}\ , \ \ 
g_*([\mathcal{E}]\cdot\beta^n)= \sum_{i=0}^{\infty} (-1)^i [R^i g_*(\mathcal{E})]\cdot \beta^{n}\ .$$
Here $d$ represents the relative dimension of the smooth equi-dimensional morphism $f$. In both these theories the general notion of fundamental class can be extended to include equi-dimensional schemes as well: for any $d$-dimensional scheme $X$ with irreducible components $X_1,\tred, X_n$ one sets 
$$i)\ [X]_{CH_*}:=\sum_{i=1}^n m_i[X_i]\ ;
 \quad  ii)\ [X]_{G_0[\beta,\beta^{-1}]}:=[\mathcal{O}_X]\cdot\beta^d\ .$$
In the first formula the coefficients $m_i$ are given by the length of the local rings $\mathcal{O}_{X,X_i}$.
 
 Provided one switches to cohomological notations, the restriction of these two functors to $\SM$ yields two examples of OCTs: the Chow ring $CH^*$ and a graded version of the Grothendieck ring of locally free sheaves which we will denote by $K^0[\beta,\beta^{-1}]$. Notice that in this case $\text{deg}\, \beta=-1$ and that pull-backs and push-forwards are given by
  $$f^*([\mathcal{E}]\cdot\beta^n)=[f^*(\mathcal{E})]\cdot \beta^n\ ; \ \ 
g_*([\mathcal{E}]\cdot\beta^n)= \sum_{i=0}^{\infty} (-1)^i [R^i g_*(\mathcal{E})]\cdot \beta^{n-d}\ ,$$
with $d$ representing the pure codimension of the projective morphism $g$.

\paragraph*{Chern classes and formal group laws:}
    As an immediate consequence of the projective bundle formula every OBM allows a theory of Chern classes operators $\widetilde{c_i}$ which, when restricted to the corresponding OCT, gives rise to a theory of Chern classes $c_i$. For sake of simplicity we will most often deal only with Chern classes, but the reader should be aware that all statements have a counterpart on the operator side. As it was pointed out in the introduction, in this general context it is no longer true that first Chern classes are linear with respect to the tensor product of line bundles. Instead, for any theory $A^*$ there exists two power series $F_A\in A^*(k)[[u,v]]$ and $\chi_{F_A}\in A^*(k)[[u]]$ such that for any line bundles $L$ and $M$ one has 
$$c_1(L\otimes M)=F_A(c_1(L),c_1(M))\ \ ,\ \ c_1(L^\vee)=\chi_{F_A}(c_1(L))\ .$$  
In fact the pair $(A^*(k),F_A)$ represents a commutative formal group law of rank one and $\chi_{F_A}$ is its inverse, i.e. $F_A(u,\chi_{F_A}(u))=0$. Let us recall that on a ring $R$ there always exists two basics examples of formal group laws  $(R,F_a)$ and, for any choice $b\in R$, $(R,F_m)$. They are given by
$$i)\ F_{a}(u,v)=u+v\quad\  ,\ \ \chi_{F_a}(u)=-u\ \ ; 
\quad ii)\  F_{m}(u,v)=u+v-buv\ \ , \ \chi_{F_m}(u)=\frac{-u}{1-bu}\ ;$$
 and are respectively referred to as \textit{additive} and \textit{multiplicative}. Moreover, if $b$ happens to be invertible $(R,F_m)$ is said to be \textit{periodic}. These two formal group laws are precisely the ones associated to the Chow ring and to $K^0[\beta,\beta^{-1}]\,$: for $CH^*$ one has $(\ZZ,F_{a})$ while for the Grothendieck ring one obtains  $(\ZZ[\beta,\beta^{-1}],F_{m})$ with $b=\beta$, as it easily follows from the formula $c_1(L)=1-[L^\vee]\beta^{-1}$. 

In \cite{Lazard} Lazard showed that there exists a formal group law $(\Laz, F_\Laz)$ which is universal: for any other law $(R,F_R)$ there exists a unique morphism $\varPhi_{(R,F_R)}:\Laz\rightarrow R$ which maps the coefficients of $F_\Laz$ onto those of $F_R$. In particular this can be applied to the formal group law arising from an OCT $A^*$ and we will denote the corresponding morphism by $\varPhi_A$. More in detail, the Lazard ring $\Laz$ is isomorphic to the polynomial ring $\ZZ[x_i]$ with $i\in \NN\setminus\{0\}$, while the universal formal group law is given by 
$$F_{\Laz}(u,v)=\sum_{i,j} a_{i,j}u^i v^j\ ,$$
where one canonically identifies $x_1$ with $a_{1,1}$ and all other coefficients $a_{i,j}$ can be expressed in terms of the first $i+j-1$-th variables. With this identification at hand it is easy to describe the maps arising from the formal group laws associated to the Chow ring and to $K^0$: $\varPhi_{CH}$ maps to zero all the variables $x_i$, while $\varPhi_{K^0[\beta, \beta^{-1}]}$ only differs at $x_1$, which is mapped to $\beta$.  From now on we will drop the underscript $_\Laz$ when writing both the universal formal group law and its inverse. Finally, let us observe that the Lazard ring can be given a grading by setting either $\text{deg}\,x_i=i$ or $\text{deg}\,x_i=-i\,$: we will denote the corresponding graded rings by $\Laz_*$ and $\Laz^*$.
   
\paragraph*{Algebraic cobordism:}
The main achievement in \cite{AlgebraicLevine} is the construction of algebraic cobordism, an OBM on $\SCH$ denoted by $\Omega_*\,$ which the authors prove to be universal among such theories and whose associated OCT on $\SM$ $\Omega^*$ is universal as well (\cite[Theorem 7.3]{AlgebraicLevine}). In other words, for every OBM $A_*$ there exists a unique morphism of OBMs $\vartheta_{A_*}: \Omega_*\rightarrow A_*$, i.e. a natural transformation of functors compatible with both l.c.i. pull-backs and the external products. Similarly one has a unique morphism $\vartheta_{A^*}$ for the corresponding OCTs. This property characterizes $\Omega^*$ as the exact algebro-geometric analogue of Quillen's complex cobordism $MU^*$, a fact that is reflected in the associated formal group laws as well: in both cases the coefficient ring is nothing but the Lazard ring and the formal group law is the universal one. 

Following the original description given in \cite{AlgebraicLevine}, instead of the more recent developed in \cite{LevineRevisited}, a cobordism cycle on $X$ can be written as 
$$[f:Y\rightarrow X,L_1,\tred, L_r]\ ,$$ 
where $f$ is a projective morphism and  $L_1,\tred,L_r$ are line bundles over $Y\in\SM$. Roughly speaking $\Omega_*$ is obtained from the free group generated by isomorphism classes of cobordism cycles  by taking successive quotients, each of which imposes a different geometric condition: the dimension axiom (Dim), the section axiom (Sect) and the formal group law axiom (FGL). The most important technical result, which is used to prove that $\Omega_*$ satifies the extended homotopy property as well as the projective bundle formula, is the existence of a short localization sequence 
$$\Omega_*(Z)\stackrel{i_*}\longrightarrow \Omega_*(X)\stackrel{j^*}\longrightarrow \Omega_*(U)\longrightarrow 0 \ $$
for any closed embedding $i:Z\rightarrow X$ with open complement $j:U\rightarrow X$. Finally one needs to define pull-back maps for l.c.i. morphisms: this is achieved by making use of the deformation to the normal cone, adjusting to the context the approach used by Fulton for the Chow ring in \cite{IntersectionFulton}. In this way not only $\Omega_*$ acquires the structure of an OBM, but one also obtains a product structure on its restriction to smooth schemes.
  
An important feature of algebraic cobordism is represented by the possibility to relate it to other theories thanks to its universality: for any choice of a formal group law $(R,F_R)\,$ the OBM 
$$\Omega^{(R,F_R)}_*:=\Omega_*\otimes_{\Laz} R$$
 arising from $\varPhi_{(R,F_R)}$ is universal among the theories with the chosen formal group law. Given an OBM $A_*$ it can be interesting to compare it with $\Omega_*^{(A_*(k),F_A)}$ and see whether or not the unique morphism of OBMs 
$$\vartheta^{(A_*{(k)},F_A)}_{A_*}:\Omega_*^{(A_*(k),F_A)}\rightarrow A_*$$
 is an isomorphism. Note the same problem can also be phrased in terms of OCTs. This perspective allowed Levine and Morel to relate algebraic cobordism to the Chow ring and the Grothendieck ring of locally free sheaves: they proved that the canonical morphisms
\begin{align}
i)\ \vartheta_{CH_*}^{(\ZZ,F_a)}:\Omega_*^{(\ZZ,F_a)}\rightarrow CH_*\ \ ;
\ \ ii)\ \vartheta^{(\ZZ[\beta,\beta^{-1}],F_m)}_{K^0[\beta,\beta^{-1}]}:\Omega^*_{(\ZZ[\beta,\beta^{-1}],F_m)}\rightarrow K^0[\beta,\beta^{-1}]\ \ ; \label{eq canonical}
\end{align}
are  respectively isomorphisms of OBMs and OCTs (\cite[Theorem 7.1.4]{AlgebraicLevine}). The second isomorphism was later extended to the OBM $G_0[\beta,\beta^{-1}]$  by Dai (\cite[Theorem 2.2.3]{ThesisDai}).

We close our general discussion on $\Omega_*$ by recalling a result that will be used in our main proof.

\begin{lemma} \label{lem top Chern}
Let $p:E\rightarrow X$ be a vector bundle of rank $d$ on $X\in\SM$. Suppose that $E$ has a section $s:X\rightarrow E$ such that the zero-subscheme of $s$, $i:Z\rightarrow X$ is a regularly embedded closed subscheme of codimension $d$. Then $c_d(E)=i_*[Z]_{\Omega^*}=[i:Z\rightarrow X]$.
\end{lemma}

\begin{proof}
The lemma is a restatement of \cite[Lemma 6.6.7]{AlgebraicLevine} for the special case of a smooth scheme. 
\end{proof}

\paragraph*{Connective $K$-theory}
 Let us now focus on $CK_*:=\Omega_*^{(\ZZ[\beta],F_m)}$ the OBM obtained from algebraic cobordism imposing the multiplicative formal group law. In view of the isomorphisms in (\ref{eq canonical}), this theory specializes to both $CH_*$ and $G_0[\beta,\beta^{-1}]$. 
Using methods from motivic homotopy theory Dai and Levine endowed this theory with a well-behaved fundamental class. In \cite[Corollary 6.4]{ConnectiveDai} they prove that for any equi-dimensional scheme $X$ of dimension $d$ the map 
\begin{align}
\psi_X:CK_d(X) \rightarrow G_0[\beta,\beta^{-1}]_d(X) \label{eq isomorphism}
\end{align}
obtained by restricting $\vartheta_{G_0[\beta,\beta^{-1}]}^{CK}$ to the $d$-th graded component is an isomorphism. This result allows them to define the fundamental class of $X$ as
\begin{align}
[X]_{CK}:=\psi_X^{-1}([X]_{G_0[\beta,\beta^{-1}]})  \label{eq fundamental}\ .
\end{align}
Moreover, they show that this assignment is compatible with l.c.i. pull-backs (\cite[Theorem 7.4]{ConnectiveDai}) and that it specializes to the fundamental classes of $CH_*$ and $G_0[\beta,\beta^{-1}]$ under the morphisms $\vartheta_{CH}^{CK}$ and $\vartheta_{G_0[\beta,\beta^{-1}]}^{CK}$ (\cite[Proposition 7.5]{ConnectiveDai}).  

Building on these results we now prove a lemma that allows us to relate the fundamental classes of Schubert varieties and those of their desingularisation.

\begin{lemma}  \label{lem class}
Let $X\in\SCH$ be equi-dimensional and with at worst rational singularities. Let $f:Y\rightarrow X$ be a resolution of singularities of $X$. Then $f_*([X]_{CK})=[Y]_{CK}$.  
\end{lemma}

\begin{proof}
Let $d$ be $\text{dim}_k X=\text{dim}_k Y$ and let us consider the commutative diagram resulting from applying the morphism of OBMs $\vartheta_{G_0[\beta,\beta^{-1}]}^{CK}$ to $f:Y\rightarrow X$. By restricting our attention to the $d$-th graded component of this diagram we have that the morphisms $\psi_X$ and $\psi_Y$ of (\ref{eq isomorphism}) are isomorphisms. Hence one is reduced to prove the analogous equality for the $G_0[\beta,\beta^{-1}]$ fundamental classes, which immediately follows from the fact that X has rational singularities.
\end{proof}



\paragraph*{Computations with Chern roots}
For the reader's convenience we finish this section by presenting the adaptation of some well-known computations with Chern roots to the more general context of OCTs. Let us recall that to a vector bundle $E\rightarrow X$ of rank $n$ one associates its \textit{Chern polynomial} by setting $c_t(E):=\sum_{i=0}^n c_i(E) t^i \in A^*(X)[t]$ and that the leading term is referred to as the \textit{top Chern class}. Given a factorization of the Chern polynomial into linear terms  $c_t(E)=\prod_{i=1}^n (1+y_i t)$ the elements $y_i$ are the \textit{Chern roots} of $E$. Such a factorization can be always achieved by the so-called splitting principle by passing to the full flag bundle $\flag \,E$, in which case $y_i\in A^*(\flag\, E)$, but all that is necessary is that $E$ has a full flag: one can take the $y_i$ to be the first Chern classes of the line bundles arising from the filtration. Hence if one assumes that $E$ already comes equipped with a full flag there is no need to consider $A^*(\flag\, E)$ and the Chern roots can be directly thought as elments of $A^*(X)$. Note in particular that the top Chern class is simply the product of all the Chern roots. 
    
We now want to examine what is the effect on the Chern roots and more specifically on the top Chern class of two operations: taking the dual and taking the tensor product. It is well-known that in the Chow ring (see for instance \cite[Remark 3.2.3\ (a)-(b)]{IntersectionFulton}) one obtains 
$$c_t(E^\vee)=\prod_{i=1}^{n}(1-y_i t)\quad, \quad c_t(E\otimes F)=\prod_{i=1}^n\prod_{j=1}^m (1+(x_i+y_j)t) \ ,$$
where $x_i$ represents the Chern roots of $F$. Since they essentially follow from the Whitney formula, these equalities also hold in any OCT $A^*$, provided one replaces the sum with the formal sum $F_A$ and the inverse with the formal inverse $\chi_A$. More precisely one has
$$c_t(E^\vee)=\prod_{i=1}^{n}(1+\chi_A(y_i) t)\quad, \quad 
c_t(E\otimes F)=\prod_{i=1}^n\prod_{j=1}^m (1+F_A(x_i,y_j)t)\ .$$ 
These two formulas can be assembled to obtain the following lemma.
\begin{lemma} \label{cor Chern}
Let $E$ and $F$ be two vector bundles over $X$ respectively of rank $e$ and $f$. Let $E\unddot=(E_1\subset E_2\subset ... \subset E_e=E)$ and   $F\unddot=(F=F_f\srarrow F_{f-1}\srarrow \tred \srarrow F_1)$ be full flags of $E$ and $F$ respectively. Set $y_j=c_1(E_j/E_{j-1})$ and  $x_i=c_1({\rm Ker}(F_i\srarrow F_{i-1}))$ for $j\in\{1,\tred, e\}$, $i\in\{1,\tred, f\}$. Then for any OCT $A^*$ the Chern polynomial and the top Chern class of $E^\vee\otimes F$ are given by:
$$c_t(E^\vee\otimes F)=\prod_{i=1}^f\prod_{j=1}^e (1+F(x_i,\chi_A(y_j))t) \quad, \quad c_{ef}(E^\vee\otimes F)=\prod_{i=1}^f\prod_{j=1}^e F_A(x_i,\chi(y_j))\ .$$
\end{lemma}

\section{The classical cases: $CH^*$ and $K^0$}

The aim of this section is to briefly present the background material  on which the two proofs of the Thom-Porteous formula rely. We begin by introducing the algebraic side of the picture represented by the double Schubert and Grothendieck polynomials of Lascoux and Sch{\"u}tzenberger. We will present these two families through a unification proposed by Fomin and Kirillov. We continue by describing the geometric setup which includes degeneracy loci for filtered bundles, Schubert varieties and Bott-Samelson resolutions. 
We conclude by giving an account of the sequence of intermediate steps that leads towards the formula.

\subsection{Double Schubert, Grothendieck and $\beta$-polynomials}

Double Schubert polynomials $\mathfrak{S}_\omega$ and double Grothendieck polynomials $\mathfrak{G}_\omega$ are two families of polynomials over $\ZZ$, both indexed by elements of the symmetric group. In \cite{GrothendieckFomin} Fomin and Kirillov unified these two families in the double $\beta$-polynomials, which are defined over $\ZZ[\beta]$ and specializes to the other two kinds for appropriate values of $\beta$. 

    Before we proceed with the definition, let us fix some notations. For $n\in\NN$, we will denote by $s_i$ the $i$-th elementary transposition $(i\ i+1)\in S_n$ and by $l$ the length function: $l(\omega)$ is defined as the minimal number of elementary transpositions needed to express $\omega\in S_n$ and it also coincides with the number of its inversions.
 Note that on $S_n$ the length function achieves a maximum at $\omega_0^{(n)}=(1\,n)(2\,n-1)\trecd (\lfloor n/ 2\rfloor\,\lceil n/2\rceil)$. We will generally drop the superscript $^{(n)}$ unless there is an ambiguity on the ambient symmetric group. Given an $l$-tuple $I=(i_1,\tred,i_l)$ of indices $i_j\in\{1,\tred, n-1\}$, we will denote by $s_I$ the product $s_{i_1}\tred s_{i_l}$ and we will say that $I$ is a minimal decomposition of $s_I$ if $l=l(s_I).$ If we denote by $T\subseteq \NN^3$  the set of triples $(t_1,t_2,t_3)$ such that $t_3\leq\min(t_1,t_2)$, then for any $\textbf{t}=(s,t,u)\in T$ we define $\nu_{\textbf{t}}\in S_{s+t-u}$ as
$$\nu_{\textbf{t}}:=
\left( \begin{array}{ccccccccc}
1 & \trecd & u & u+1&  \trecd &  t & t+1& \trecd & s+t-u \\
1 & \trecd & u &  s+1& \trecd & s+t-u & u+1 & \trecd & s
\end{array} \right)\ .$$
 As a general principle we will shorten the dependence by a group of variables by using bold letters.

 The definition of double $\beta$-polynomials $\mathfrak{H}^{(\beta)}_\omega(\textbf{x},\textbf{y})$ is given recursevely: one sets the polynomial associated to the longest permutation $\omega_0$ and the other members of the family are obtained from it by means of the so-called divided difference operators. For each $i\in\{1,\tred,n-1\}$ the $\beta$-divided difference operator~$\phi^{(\beta)}_i$ on $\ZZ[\beta][\bf{x},\bf{y}]$ are defined by setting 
\begin{align}\label{def H1}
\phi^{(\beta)}_i P=(1+\sigma_i)\frac{(1+\beta x_{i+1})P}{x_i-x_{i+1}}=\frac{(1+\beta x_{i+1}) P-(1+\beta x_{i})\sigma_i (P)}{x_i-x_{i+1}}\ ,
\end{align}
where $\sigma_i$ is the operator exchanging $x_i$ and $x_{i+1}$ and $1$ represents the identity operator. The initial element of the recursion is given by
\begin{align}\label{def H2}
\mathfrak{H}^{(\beta)}_{\omega_0}:=\prod_{i+j\leq k}(x_i+y_j+\beta x_i y_j)\ .
\end{align}
For all other elements of $S_n$ there exists an elementary transposition $s_i$ such that $l(\omega)<l(\omega s_i)$ and for those one sets 
\begin{align}\label{def H3}
\mathfrak{H}^{(\beta)}_\omega:=\phi^{(\beta)}_i \mathfrak{H}^{(\beta)}_{\omega s_i}\ .
\end{align} 
Double Schubert and Grothendieck polynomials, together with the corresponding divided difference operators, are then recovered through the following formulas:
$$i)\  \mathfrak{S}_\omega:=\mathfrak{H}^{(0)}_\omega(x_1,\tred,x_n,-y_1,\tred,-y_n) \ ,  \ \ \partial_i:=\phi^{(0)}_i\ ;
\ \  ii)\  \mathfrak{G}_\omega:=\mathfrak{H}^{(-1)}_\omega\ ,\ \ \pi_i:=\phi^{(-1)}_i\ .$$

From now on we will skip the superscript $^{(\beta)}$ for both the operators and the polynomials, unless we set $\beta$ equal to a specific value.

\begin{remark} \label{rem braid}
 A priori it is not evident neither that this is a good definition, since different decompositions of $\omega_0\omega$ in elementary transpositions might yield different polynomials, nor that $\mathfrak{H}_\omega$ should be independent of the symmetric group $\omega$ belongs to. While the first point follows from the fact that the operators $\phi_i$ satisfy the braid relations, hence allowing to define operators $\phi_\omega$ independent of the decomposition, the second issue can be settled by verifying that $\mathfrak{H}_{\omega_0^{(n)}}$ is unchanged if $\omega_0^{(n)}$ is viewed as an element of $S_{n+1}$. 
\end{remark}

\begin{remark}
In order to recover the original definition of double Grothendieck polynomials as it is given in \cite{PieriFulton}, it is necessary to allow all variables to be invertible and make the following change of coordinates: $x_i \mapsto 1-x_i^{-1}\ ,\  y_i \mapsto 1-y_i$. 
\end{remark}

In \cite{GrothendieckFomin} the authors also propose an alternative definition of $\beta$-polynomials, viewed as coefficients of a particular element in a Hecke algebra and this approach proves to be well suited to extend some known properties of double Schubert polynomials. We refer the interested reader to the appendix, in which the construction is reviewed, together with the proof of the equivalence of the two definitions and some immediate consequences. The main point for our discussion is that one has the following

\begin{lemma}[Appendix, lemma \ref{lem symmetry}]
If $\omega(i)<\omega(i+1)$ then $\mathfrak{H}^{(\beta)}_{\omega}$ is symmetrical in $x_i$ and $x_{i+1}$. Also, if $\omega^{-1}(i)<\omega^{-1}(i+1)$, then $\mathfrak{H}^{(\beta)}_{\omega}$ is symmetrical in $y_i$ and $y_{i+1}$.
\end{lemma}

It follows from this lemma that for any $\mathbf{t}\in T$ the polynomials $\mathfrak{H}_{\nu_{\textbf{t}}}$ are separately symmetric in $x_1$, $\tred\,$, $x_t$ and in $y_1$, $\tred\,$, $y_s$. This implies that if we set the remaining variables equal to 0, we can express these polynomials by means of the elementary symmetric functions $e_i(\textbf{x})$ and $e_i(\textbf{y})$: 
\begin{align}
\mathfrak{D}^{(\beta)}_{\textbf{t}}(e_1(\textbf{x}),\tred, e_t(\textbf{x}),e_1(\textbf{y}),\tred,e_s(\textbf{y})):=
\mathfrak{H}^{(\beta)}_{\nu_{\textbf{t}}}(x_1,\tred,x_t,0,\tred,0,y_1,\tred,y_s,0,\tred,0) \label{def symmetric}
\end{align}
We define $\mathfrak{D}^{CH}_{\textbf{t}}$ and $\mathfrak{D}^{K^0}_{\textbf{t}}$ to be the polynomials obtained by specializing the value of $\beta$ to 0 and~-1.
\subsection{Degeneracy loci, Schubert varieties and Bott-Samelson resolutions}

 Given a morphism $h:E\rightarrow F$ between vector bundles respectively of rank $e$ and $f$ over $X\in \SM$, we can associate to any $r$, with $0\leq r\leq min(e,f)$, the $r$-th degeneracy locus 
$$D_r(h):=Z(\wedge^{r+1}h)=\{x\in X\ |\ {\rm rank}(h(x):E(x)\rightarrow F(x))\leq r\}\ , $$
i.e. the zero scheme of the section of $\text{Hom}(\wedge^{r+1}E,\wedge^{r+1}F)$ corresponding to $\wedge^{r+1}h\,$. 
The notion of degeneracy locus can be conveniently extended to the case of vector bundles equipped with full flags. Given $E\unddot=(E_1\subset\tred \subset E_e=E)$  a full flag of subbundles of $E$, $F\unddot=(F=F_f\srarrow\tred\srarrow F_1)$ a full flag of quotient bundles of $F$ and set of rank conditions $\textbf{r}:\{1,\tred,e\}\times\{1,\tred, f\}\rightarrow \NN$, set
$$\Omega_\textbf{r}(E_{\textbf{\textbullet}},F_{\textbf{\textbullet}},h):=\bigcap_{(i,j)} D_{\textbf{r}(i,j)}(h_{ij})=\{x\in X \ |\ {\rm rank}(h_{ij}(x):E_i(x)\rightarrow F_j(x))\leq \textbf{r}(i,j)\ \forall i,j \}\ .$$
Let us recall that to every full flag, either of subbundles or of quotient bundles, it is possible to associate a filtration into linear factors: for $F\unddot$ one considers $\{L_i^{F\unddot}\}_{i\in\{1,\tred,n\}}$ with 
$L^{F\unddot}_i:={\rm Ker\,}(F_i\srarrow F_{i-1})$ and for $E\unddot$ one takes $\{L_i^{E/E\unddot}\}_{i\in\{1,\tred, n\}}$, where $E/E\unddot$ is the obvious full flag of quotient bundles obtained from~$E\unddot$.


In our discussion we will restrict our attention to a particular family of sets of rank conditions, indexed by the elements of the symmetric group. For $\omega\in S_n$ we set
$$\textbf{r}_\omega (i,j):=|\{l\leq j \  |\ \omega (l)\leq i \}|\ .$$ 

\begin{remark}\label{remark locus}
The family $\{\textbf{r}_\omega\}_{\omega\in S_n}$ is general enough to allow to express in the form $\Omega_{\textbf{r}_\omega}(E'\unddot,F'\unddot,h')$ all degeneracy loci $D_r(h)$, provided $E$ and $F$ are already equipped with full flags.  To achieve this one sets $E':=E\oplus \mathbb{A}_X^{f-r}$, $F':=F\oplus \mathbb{A}_X^{e-r}$ and defines $h':E'\rightarrow F'$ by extending $h$ by 0 on $\mathbb{A}^{f-r}_X$. The full flags of $E'$ and $F'$ are obtained by extending those of $E$ and $F$ with trivial line bundles. Finally, for the permutation one sets
$w=\nu_{\textbf{t}}$ with $\textbf{t}=(e,f,r)$. It is possible to check that both degeneracy loci are defined by the same equations.
\end{remark}

\begin{remark}\label{remark id}
It is worth noting that it is not restrictive to assume that the two vector bundles coincide and that the morphism $h$ is just the identity: both bundles can be replaced by their direct sum. By considering the embedding $E\stackrel{id_E\oplus h}{\longrightarrow} E\oplus F$ and the projection $pr_2:E\oplus F\rightarrow F$, it is easy to extend the given flags to $E'\unddot$ and $F'\unddot$ in such a way that $F'_i=(E\oplus F)/E'_{2n-i}$ and it can be checked that $\Omega_{\textbf{r}_\omega}(E\unddot,F\unddot,h)$ and $\Omega_{\textbf{r}_\omega}(E'\unddot,F'\unddot,id_{E\oplus F})$ are defined by the same equations. 
\end{remark}

We now leave the generality of degeneracy loci over any scheme and we focus on the universal case represented by Schubert varieties. For this we need first to introduce the full flag bundle. Given a vector bundle $V\rightarrow X$ of rank $n$, we will denote by $\flag V\stackrel{\pi}\rightarrow X$ the full flag bundle of $V$. This scheme, which can be explicitly constructed as an iterated projective bundle, is the relative version of the flag manifold and, precisely as the flag manifold, comes equipped with a universal full flag of quotient bundles $Q\unddot=(V=Q_n\srarrow Q_{n-1}\srarrow \tred\srarrow Q_1)$ and has a defining universal property: for any morphism $f:Y\rightarrow X$ and for any choice $W\unddot$ of a full flag of quotient bundles of $f^*V$, there exists a unique morphism $\tilde{f}$ such that $f=\pi\widetilde{f}$ and $\widetilde{f}^*Q\unddot=W\unddot\,$. Note in particular that this yields a section $i_{W\unddot}:=\widetilde{id_V}:X\rightarrow \flag V$ for every full flag of quotient bundles of $V$.
 
It is also possible to consider partial flag bundles and in our treatment we will specifically need the ones  parametrizing flags in which only the $i$-th level is missing: we will denote them by $\flag_{\widehat{i}}V$. It is easy to see that $\flag V$ can be recovered  from each of these bundles $\flag_{\widehat{i}}V$'s  as a $\mathbb{P}^1$-bundle and we will write $\varphi_i$ for the corresponding projection.
    
Schubert varieties $\Omega_{\omega}$ are defined as degeneracy loci of $\flag V$ and their definition depends on the choice of  a full flag of subbundles $V\unddot=(V_1\subset\tred\subset V_n=V)$. More precisely one sets
$$\Omega_\omega:=\Omega_{\textbf{r}_\omega}(\pi^* V\unddot,Q\unddot,id_{\pi^* V})\ .$$ 

\begin{remark}\label{remark pullback}
Each degeneracy locus of the kind $\Omega_{\textbf{r}_\omega}(U\unddot,W\unddot,id_V)$ can be obtained as a pull-back along $i_{W\unddot}$ of the Schubert variety $\Omega_\omega$ arising from $U\unddot$. 
\end{remark}

It is important to notice that although in general Schubert varieties are not neither smooth nor l.c.i. schemes, the Schubert variety $\Omega_{\omega_0}$, being isomorphic to the base scheme $X$, is always smooth. This observation is central in the definition of Bott-Samelson resolutions, a family of smooth schemes over $\flag V$, which in some sense represents the missing  link between Schubert polynomials and Schubert varieties. If on the one hand, as we will see in the next section, double Schubert polynomials naturally describe the push-forwards of the fundamental classes of Bott-Samelson resolutions, on the other hand each Schubert variety is birationally isomorphic to at least one element of the family.     

 Exactly as for Schubert polynomials, the definition of Bott-Samelson resolutions $\mathcal{R}_I\stackrel{r_I}\rightarrow\flag V$ is given recursively, even though in this case the indexing set is represented by tuples of indices $i_j\in\{1,\tred, n-1\}$. For $I=\emptyset$ one sets $R_\emptyset:=\Omega_{\omega_0}$ and $r_\emptyset:=i_{V/V_{\textbf{\textbullet}}}$. In all other cases it is possible to write $I=(I',j)$, to consider the following fibre diagram
\begin{eqnarray}
\xymatrix{
  R_{I'}\times_{\flag_{\hat{j}}V} \flag(V) \ar[rr]^{pr_2} \ar[d]_{pr_1}& &\flag(V)\ar[d]^{\varphi_j} \\
  R_{I'}\ar[r]^{r_{I'}}& \flag(V) \ar[r]^{\varphi_j} &\flag_{\widehat{j}}V    }
\end{eqnarray}
and set $R_I:=R_{I'}\times_{\flag_{\hat{j}}V} \flag(V)$ and $r_I:=pr_2$. The relationship existing between Bott-Samelson resolutions and Schubert varieties is made explicit by the following results.

\begin{proposition}\label{prop resolution}

Let $I=(i_1,\tred, i_l)$ be a minimal decomposition and set $\omega=\omega_0 s_I$. Then 

1) $r_I(R_I)=\Omega_{\omega}$ and  the resulting map $R_I\rightarrow \Omega_{\omega}$ is a projective birational morphism. In particular if $X\in \SM$, then $R_I$ is a resolution of singularities of $\Omega_{\omega}$;

2) i) $r_{I*}\mathcal{O}_{R_I}=\mathcal{O}_{\Omega_\omega}$  as coherent sheaves and therefore $\Omega_\omega$ is a normal scheme;

\ \ \ ii) $R^q f_*\mathcal{O}_{R_I}=0$ for q>0, hence $\Omega_\omega$ has at worst rational singularities. 
\end{proposition}

\begin{proof}
For part (1) see \cite[Appendix C]{SchubertFulton}. For part (2) see \cite[Theorem 4]{SchubertRamanathan}.
\end{proof}
\subsection{Thom-Porteous formulas for $CH^*$ and $K^0$}
Is it worth stressing that every result in this section has an exact counterpart for the case of a general oriented cohomology theory, from which it can be derived. Our exposition of the proof of the formula will emphasise the role played by Bott-Samelson resolutions because we believe this allows to gain a better understanding of the situation in the case of a general oriented cohomology theory. The first step towards the formula consists in understanding how to express the fundamental class of a Schubert variety $\Omega_{\omega}\subseteq \flag V$ in both $CH^*(\flag V\,)$ and $K^0(\flag V\,)$ and for this it is necessary to have an explicit description of these two rings.

\begin{lemma}\label{prop flag CH}
Let $J_{CH}$ and $J_{K^0}$ be the ideals of $CH^*(X)[\textbf{X}]$ and $K^0(X)[\textbf{X}]$ generated by the elements $e_i(\textbf{X})-c_i(V)$, where, for $1\leq i\leq n$, $e_i(\textbf{X})$ is the $i$-th elementary symmetric function and $c_i(V)$ is the $i$-th Chern class of $V$. Then one has 
$$\ CH^*(\flag V\,)\simeq CH^*(X)[\textbf{X}]/J_{CH}\quad, \quad K^0(\flag V\,)\simeq K^0(X)[\textbf{X}]/J_{K^0}\ ,$$
where the isomorphisms map the $X_i$'s to the Chern roots of $\pi^*V$ associated to $Q\unddot$. 
\end{lemma}

 
At this point one needs to establish a relation between double Schubert polynomials and Bott-Samelson resolutions. If we denote by  $\mathcal{R}_I^{CH}$  the push-forward along $r_I$ of the fundamental class of a given resolution $R_I\stackrel{r_I}\rightarrow\flag V$, it immediately follows from the definition that $\mathcal{R}^{CH}_I=\varphi_{j*}\varphi^*_j \mathcal{R}^{CH}_{(I,j)}$ and the same holds as well for  $\mathcal{R}^{K^0}_I$. As a consequence, in order to be able to express these classes it is sufficient to have an explicit description of the operators $\varphi_{j*}\varphi_j^*$ and of the initial class $\mathcal{R}_\emptyset$.

\begin{lemma}
Denote by $\overline{\partial_i}$, $\overline{\pi_i}$ the operators defined by $\partial_i$, $\pi_i$ on $CH^*(\flag V)$ and $K^0(\flag V)$ respectively. Then one has $\overline{\partial_i}=\varphi_{i*}\varphi_i^*$ and $\overline{\pi_i}=\varphi_{i*}\varphi_i^*$.
\end{lemma}    

\begin{lemma} Denote by $M_i$ and $L_i$ the line bundles $L_i^{Q\unddot}$ and $L_i^{\pi^*V\unddot}$. Then
$$i)\ \mathcal{R}_\emptyset^{CH}=\mathfrak{S}_{\omega_0}(c_1(M_k),c_1 (L_j)) \ \ ;\ \  \
ii)\ \mathcal{R}_\emptyset^{K^0}=\mathfrak{G}_{\omega_0}(c_1 (M_k),c_1 (L_j^\vee))\ .$$
\end{lemma}

In fact these two lemmas together yield the following
\begin{corollary}
Let $I$ be a minimal decomposition and set $\omega=\omega_0s_I$. Then
$$i)\ \mathcal{R}_I^{CH}=\partial_{s_I}\mathfrak{S}_{\omega_0}(c_1(M_k),c_1(L_j))=\mathfrak{S}_{\omega}(c_1(M_k),c_1(L_j))\ ;$$
$$ii)\ \mathcal{R}_I^{K^0}=\pi_{s_I}\mathfrak{G}_{\omega_0}(c_1(M_k),c_1(L_j))
=\mathfrak{G}_{\omega}(c_1(M_k),c_1(L_j^\vee))\ ;$$
where $\partial_{s_I}$ and $\pi_{s_I}$ represent the obvious composition of divided difference operators associated to $I$.
\end{corollary}

Hence double Schubert and Grothedieck polynomials naturally describe the classes $\mathcal{R}_I$. 

\begin{remark}
It is worth pointing out that the lack of symmetry between the two statements, namely the presence of an extra dual in the formula for $K^0$, is only due to the fact that double Schubert polynomials already incorporate the dual in their definition. 
\end{remark}

 In order to complete the treatment of the universal case we only need to use proposition \ref{prop resolution} to relate these quantities to the fundamental classes of Schubert varieties.

\begin{proposition}[\protect{\cite[Proposition 7.5]{FlagsFulton}}, \protect{\cite[Theorem 3]{PieriFulton}}]  \label{th flagbundle}

Let $V\rightarrow X$ be a vector bundle and let $\omega\in S_n$. Denote by $M_i$ and $L_i$ the line bundles $L_i^{Q\unddot}$ and $L_i^{\pi^*V\unddot}$.
 In $CH^*(\flag(V))$ and in $K^0(\flag V)$ one respectively has
$$i)\ [\Omega_\omega]_{CH^*}=\mathfrak{S}_\omega(c_1(M_i),c_1 (L_j))\ ;
\ \  ii) \ \ [\mathcal{O}_{\Omega_\omega}]_{K^0}=\mathfrak{G}_\omega(c_1(M_i),c_1 (L_j^\vee))\ .$$ 
\end{proposition}

 In view of remarks \ref{remark id} and \ref{remark pullback}, the universal case can then be used to give a description of the fundamental classes of degeneracy loci of morphisms of bundles with full flags. In order to express the statements of the theorems in the greatest generality, we will assume $X$ to be a Cohen-Macaulay scheme.

\begin{theorem}[\protect{\cite[Theorem 8.2]{FlagsFulton}},\protect{\cite[Theorem 2.1]{GrothendieckBuch}}] \label{thm Schubert}
Let $h:E\rightarrow F$ be a morphism of vector bundles of rank $n$ over a pure dimensional Cohen-Macaulay scheme $X$. Let $E_{\textbf{\textbullet}}$ and $F_{\textbf{\textbullet}}$ be full flags of $E$ and $F$ respectively.  Let $\omega\in S_n $ and assume that the degeneracy locus $\Omega_{r_\omega}(E\unddot,F\unddot,h)$ has codimension $l(\omega)$ in $X$. Denote by $M_i$ and $L_i$ the line bundles $L_i^{F\unddot}$ and $L_i^{E\unddot}$. Then in $CH_*(X)$ and $K^0(X)$ one has  
$$i)\ [\Omega_{r_\omega}(E\unddot,F\unddot,h)]_{CH_*}=
\mathfrak{S}_\omega(c_1(M_i),c_1(L_j))\ ;
\ \ ii)\ [\mathcal{O}_{\Omega_{r_\omega}(E\unddot,F\unddot,h)}]_{K^0}=
\mathfrak{G}_\omega(c_1(M_i),c_1(L^\vee_j))\ .$$
\end{theorem}

The Thom-Porteous formula finally follows: it suffices to make use of the splitting principle to obtain full flags for both $E$ and $F$ and to observe that, thanks to remark \ref{remark locus}, one can reduce to the previous theorem. Since the polynomials in the formulas are symmetric in the two groups of Chern roots, it follows that the result does not depend on the choice of the flags.

\begin{corollary}[Thom-Porteous formula, \protect{\cite[Theorem 14.4(c)]{IntersectionFulton}}, \protect{\cite[Theorem 2.3]{GrothendieckBuch}}] \label{cor Thom}

Let $E\stackrel{h}\longrightarrow F$ be a morphism of vector bundles of rank $e$ and $f$ and fix $r$ with $0\leq r\leq min(e,f)$. Denote by $\textbf{t}$ the triple $(e,f,r)$. Assume that $codim(D_r(h),X)=(e-r)(f-r)$ and $X$ is a Cohen-Macaulay scheme. Then $D_r(h)$ is Cohen-Macaulay and in $CH^*(X)$ and $K^0(X)$ one respectively has
$$i)\ [D_r(h)]_{CH}=\mathfrak{D}^{CH}_{\textbf{t}}(c_i(F),-c_j(E))\ ;
\ \ ii)\ [\mathcal{O}_{D_r(h)}]_{K^0}=\mathfrak{D}^{K^0}_{\textbf{t}}(c_i(F),c_j(E^\vee))\ .$$
\end{corollary}

\section{Extension to other oriented cohomology theories}
We begin this section by providing a description of the classes $\mathcal{R}_I$ for algebraic cobordism and as a consequence, for any oriented cohomology theory. In the second part we derive some consequences for connective $K$-theory, among them the Thom-Porteous formula.

\subsection{Push-forward classes of Bott-Samelson resolutions for $\Omega^*$}
As in the classical cases, one firsts needs an explicit description of the ring associated to $\flag V$.

\begin{proposition}[\protect{\cite[Theorem 2.6]{SchubertHornbostel}}]
Let $V$ be a vector bundle of rank $n$ over $X\in\SM$ and let $J$ be the ideal of $\Omega^*(X)[X_1,\tred,X_n]$ generated by the elements $e_i(\textbf{X})-c_i(V)$ where, for $1\leq i\leq n$, $e_i(\textbf{X})$ is the $i$-th elementary symmetric function and $c_i(V)$ is the $i$-th Chern class of $V$. Then   
$$\Omega^*(\flag (V))\simeq \Omega^*(X)[X_1,\tred,X_n]/J\ ,$$
where the isomorphisms map the $X_i$'s to the Chern roots of $\pi^*V$. 
\end{proposition}

\begin{proof}
Since the full flag bundle can be constructed as an iterated $\Proj^m$ bundle, it is sufficient to apply several times of the projective bundle formula and obtain $\Omega^*(\flag V)$ as the quotient of a polynomial ring over $\Omega^*(X)$. One then writes the generators in the desired form by means of some algebraic manipulations with symmetric functions.
\end{proof}

This being achieved, it is necessary to identify the push-pull operators $\varphi_i^*\varphi_{i*}$. This problem was first solved by Bressler-Evens in the topological context of complex-oriented cohomology theories in \cite{BraidBressler} and their result was later imported in the algebraic setting by Hornbostel-Kiritchenko in \cite[Corollary 2.3]{SchubertHornbostel}. 
Let us define the generalized divided difference operators $A_i$  on $\Laz[[\textbf{x},\textbf{y}]]$ by setting $$A_i(f):=(1+\sigma_i)\frac{f}{F(x_i,\chi (x_{i+1}))}\ ,$$ where $1$ represents the identity operator and $\sigma_i$ exchanges $x_i$ and $x_{i+1}$. Moreover, for a nonempty tuple of indices $I=(i_1,\tred,i_l)$, we will write $A_I$ for the composition of operators $A_{i_l}\trecd A_{i_1}$. Clearly this definition makes sense for any formal group law $(R,F_R)$: taking the tensor product with $R$ over $\Laz$ with respect to the classifying morphism $\varPhi_{(R,F_R)}$  yields operators $A_i^{(R,F_R)}$ on $R[[\textbf{x},\textbf{y}]]$.

\begin{remark} \label{rem operators}
 It should be noticed that this procedure can in particular be applied to the formal group law arising from an OCT and, more specifically, that the restriction of $A_i^{CK^*}$ to the polynomial ring $\ZZ[\beta][\textbf{x},\textbf{y}]$ returns the operators $\phi_i^{(\beta)}$.
\end{remark}   
 If one denotes by $\overline{A_i}$ the operators induced on $\Omega^*(\flag V)$, then we have the following  

\begin{proposition}\label{prop operators}
For any $i\in\{1,\tred, n-1\}$ the operator $\varphi_i^*\varphi_{i*}:\Omega^*(\flag V)\rightarrow \Omega^*(\flag V)$ coincides with $\overline{A_i}$. 
\end{proposition}

\begin{proof}
Since pull-back maps are easily described, the main point is to have an explicit expression for the push-foward map of a $\Proj^1$-bundle and this is obtained by specializing the formula for projective bundles proved by Vishik in \cite[Theorem 5.30]{SymmetricVishik}. For a detailed exposition see \cite[Section 2.1]{SchubertHornbostel}.
\end{proof}

\begin{remark}
It is important to point out that the operators $A_i$, unlike all other operators we have encountered so far, do not satisfy the braid relations. For a proof see \cite[Theorem 3.7] {BraidBressler}.
\end{remark}

In view of the recursive definition of Bott-Samelson resolutions, in order to complete the description of the cobordism classes $\mathcal{R}_I:=[R_I\stackrel{r_I}\rightarrow \flag V]=r_{I*}[R_I]_{\Omega^*}$ one only needs to give an expression for the initial class $\mathcal{R}_\emptyset$: all others can be recovered by the relation 
$\mathcal{R}_{(I,j)}=\varphi_j^*\varphi_{j*}\mathcal{R}_I$. Note that the same holds as well for $\mathcal{R}^A_I$ in any other OCT $A^*$.

\begin{proposition}\label{prop initial class}
Let $V\unddot=(V_1\subset V_2\subset ... \subset V_n=V)$ be a full flag of subbundles of $V$ and $Q\unddot=(\pi^*V=Q_n\srarrow Q_{n-1}\srarrow ... \srarrow Q_1)$ be the universal full flag of quotient bundles of $\pi^* V$. Denote by $M_i$ and $L_i$ the line bundles $L_i^{Q\unddot}$ and $L_i^{\pi^*V\unddot}$.
Then in $\Omega^*(\flag V)$ one has 
$$\mathcal{R}_\emptyset=\prod_{k+j\leq n} F(c_1(M_{k}),\chi(c_1(L_j)))\ .$$

\end{proposition}
  
\begin{proof}
The proof does not significantly differ from the one of Fulton for the Chow ring case (\cite[Proposition 7.5]{FlagsFulton}). More specifically, the geometric part is unchanged: one constructs a bundle $K$ of rank $N:=\frac{n(n-1)}{2}$, together with a section $s$, such that the zero scheme $Z(s)$ will coincide with $R_\emptyset=\Omega_{\omega_0}$. One then recovers $\mathcal{R}_I$ as the top Chern class of $K$ by lemma \ref{lem top Chern}. It is in the explicit computation of this class that the difference between the two theories appears, but this does not lead to any substantial change in the algebra required to obtain the final expression.

 More specifically one sets 
$$K:=\text{Ker}\left(\bigoplus_{l=1}^{n-1} \text{Hom}(\pi^*V_l,Q_{n-l})\stackrel{\psi}\longrightarrow \bigoplus_{l=1}^{n-2}\text{Hom}(\pi^*V_l,Q_{n-l-1})\right)\ ,$$
where $\psi$ assigns to  
$\{g_l\}_{l\in\{1,\tred, n-1\}}$ the family $\{g_{l+1}\circ i_l-p_{n-l}\circ g_l\}_{l\in\{1,\tred, n-2\}}$. Here 
$i_l:\pi^*V_l\irarrow \pi^*V_{l+1}$ and $p_l: Q_l\srarrow Q_{l-1}$ are respectively the injections and the projections within the two flags. Since $\psi$ is surjective, the Whitney formula allows to recover the Chern polynomial of $K$ as the ratio of those of the other two bundles and from this it follows immediately that the same holds for the top Chern classes as well. The explicit computation of these classes is reduced, again by the Whitney formula, to the identification of the Chern roots of bundles of the form $\text{Hom}(\pi^*V_{m_1},Q_{m_2})\simeq(\pi^* V_{m_1})^\vee\otimes Q_{m_2}$. Once this is achieved, using lemma \ref{cor Chern}, one notices that all factors in the denominator also occurs in the numerator and therefore one simply has to identify the surving ones: this yields $$c_N(K)=\prod_{k+j\leq n}F(c_1(M_{k}),\chi(c_1(L_j)))\ .$$
 To finish the proof we only need to provide a section such that its zero scheme coincide with $\Omega_{\omega_0}$:  the family of morphisms $h_{l,n-l}:\pi^*V_l\irarrow \pi^*V\srarrow Q_{n-l}$ is clearly sent to $0$ by $\psi$ and 
, as consequence, defines a section of $K$, which happens to satisfy the required condition.
\end{proof}

The analogy with the situation in the Chow ring and in Grothendieck ring with respect to the double Schubert and Grothendieck polynomials suggests the following definition.

\begin{definition}
Fix $n\in\NN$. To every tuple $I=(i_1,\tred,i_l)$ with $i_j\in\{1,\tred,n-1\}$ we associate a power series $\mathfrak{B}^{(n)}_I\in\Laz[\variables]$ by means of a recursive procedure on the length of $I$. For $I=\emptyset$ we set $$\mathfrak{B}^{(n)}_\emptyset:=\prod_{k+j\leq n} F(x_k,y_j)\ .$$
For $I\neq\emptyset$, the operator $A_I$ is well defined  and we set 
$$\mathfrak{B}^{(n)}_I:=A_I \mathfrak{B}^{(n)}_\emptyset\ .$$
The same definition can be carried out in any OCT $A^*$. We will denote the corresponding power series by 
$\mathfrak{B}^{(A,n)}_I$.
\end{definition} 

\begin{remark} \label{rem polynomial}
It is easy to verify that if $I$ is a minimal decomposition then $\mathfrak{B}^{(CK,n)}_I=\mathfrak{H}_{\omega_0s_I}^{(-\beta)}$: it was already pointed out in remark \ref{rem operators} that the divided difference operators coincide and, since the formal group laws are the same, also the starting elements are equal. 
\end{remark}

We are now ready to express our main result about algebraic cobordism, which immediately follows from  propositions \ref{prop operators} and \ref{prop initial class} .  

\begin{theorem}
Let $V\rightarrow X$ be a vector bundle of rank $n$, together with a full flag of subbundles $V\unddot$. Denote by $M_i$ and $L_i$ the line bundles $L_i^{Q\unddot}$ and $L_i^{\pi^*V\unddot}$, where $Q\unddot$ is the universal flag of quotient bundles over $\flag V\stackrel{\pi}\rightarrow X$. For any tuple $I=(i_1,...,i_l)$ with $i_j\in\{1,\tred,n-1 \}$ let us consider the associated Bott-Samelson resolution $R_I\stackrel{r_I}\rightarrow \flag V$. As an element of $\Omega^*(\flag V)$ its pushforward class is given by the formula 

$$\mathcal{R}_I=\mathfrak{B}^{(n)}_I (c_1(M_k),c_1(L_j^\vee))=
\mathfrak{B}^{(n)}_I (c_1(M_k),\chi(c_1(L_j)))\ .$$ 
\end{theorem}
    

\begin{corollary}\label{cor Bott}
Let $A^*$ be an OCT. Under the same hypothesis of the preceding theorem one has
$$\mathcal{R}^A_I=\mathfrak{B}^{(A,n)}_I (c_1(M_k),c_1(L_j^\vee))=
\mathfrak{B}^{(A,n)}_I (c_1(M_k),\chi(c_1(L_j)))\in A^*(\flag \,V)\ .$$
\end{corollary}

\begin{proof}
One only needs to apply the canonical morphism $\vartheta_{A^*}$ to the statement of the theorem.
\end{proof}

\begin{remark}
It is worth stressing that it is at this point that the analogy with the  classical case breaks down: as it was pointed out in \cite[\S 5.2]{SchubertHornbostel} even for the flag manifold $\flag\, k^3$ one has that $\mathcal{R}_{(1,2,1)}$, $\mathcal{R}_{(2,1,2)}$ and $[\Omega_{id}]_{\Omega}=[\flag\, k^3]_\Omega=1$ do not coincide. Moreover not all Schubert varieties are l.c.i. schemes, so in general they do not have a fundamental class.

\end{remark}




\subsection{Thom-Porteous formula for $CK^*$}

We will now specialise the formula of \ref{cor Bott} to the special case of connective $K$-theory. This choice is motivated by a result of Bressler-Evens (\cite[Theorem 3.7]{BraidBressler}), which says that the most general formal group law for which the braid relations hold is the multiplicative one. The point for us is that the braid relations allows us to conclude that the push-forward class $\mathcal{R}_I$ of different desingularizations of the same Schubert variety $\Omega_\omega$ are all represented by the same polynomial. 

\begin{proposition}
Let $V\rightarrow X$ be a vector bundle of rank $n$ with $X\in \SM$ and let $\omega\in S_n$. Denote by $M_i$ and $L_i$ the line bundles $L_i^{Q\unddot}$ and $L_i^{\pi^*V\unddot}$.
 In $CK^*(\flag(V))$ one has
$$[\Omega_\omega]_{CK}=\mathfrak{H}_\omega(c_1(M_i),c_1 (L^\vee_j))\ .$$ 
\end{proposition}

\begin{proof}
In view of proposition \ref{prop resolution} we can apply lemma  \ref{lem class} and hence if $I$ is a minimal decomposition of $\omega_0\omega$, then $\mathcal{R}_I^{CK}=r_{I*}[R_I]_{CK}=[\Omega_{w_0s_I}]_{CK}$. The result then follows since, as it was pointed out in remark \ref{rem polynomial}, one has $\mathfrak{B}^{(CK,n)}_I=\mathcal{H}_{\omega_0s_I}^{(-\beta)}$. 
\end{proof}

Next we consider the more general case of degeneracy loci with expected codimension. 

\begin{theorem}\label{thm connSchubert}
Let $h:E\rightarrow F$ be a morphism of vector bundles of rank $n$ over $X\in\SM$. Let $E_{\textbf{\textbullet}}$ and $F_{\textbf{\textbullet}}$ be full flags of $E$ and $F$ respectively.  Let $\omega\in S_n $ and assume that the degeneracy locus $\Omega_{r_\omega}(E\unddot,F\unddot,h)$ has codimension $l(\omega)$ in $X$. Denote by $M_i$ and $L_i$ the line bundles $L_i^{F\unddot}$ and $L_i^{E\unddot}$. Then in $CK_*(X)$ one has  
$$\ [\Omega_{r_\omega}(E\unddot,F\unddot,h)]_{CK_*}=
\mathfrak{H}_\omega(c_1(M_i),c_1(L^\vee_j))\ .$$
\end{theorem}

\begin{proof}
In view of remark \ref{remark id} we can restrict to the case in which $E=F=V$ and $h=id_V$. In this simplified setting  we can reduce to the universal case represented by Schubert varieties: thanks to the universal property of the flag bundle, the full flag of quotient bundles $F\unddot$ yields a section $s_{F\unddot}$ and the preimage $s^{-1}_{F\unddot}\Omega_\omega$ is precisely $\Omega_{r_\omega}(E\unddot,F\unddot,h)$. One then applies the previous proposition to $\Omega_\omega$ and since the degeneracy locus has the expected codimension it follows that the pull-back along  $s_{F\unddot}$ maps $[\Omega_\omega]_{CK^*}$ onto $[\Omega_{r_\omega}(E\unddot,F\unddot,h)]_{CK^*}$.
\end{proof}

We complete our treatment by establishing the Thom-Porteous formula for connective $K$-theory.

\begin{corollary}[Thom-Porteous formula]\label{cor connThom}

Let $E\stackrel{h}\rightarrow F$ be a morphism of vector bundles of rank $e$ and $f$ over $X\in\SM$ and fix $r$ with $0\leq r\leq min(e,f)$. Denote by $\textbf{t}$ the triple $(e,f,r)$ and assume that $codim(D_r(h),X)=(e-r)(f-r)$. Then in $CK^*(X)$  one  has
$$[D_r(h)]_{CK^*}=\mathfrak{D}^{CK}_{\textbf{t}}(c_i(F),c_j(E^\vee))\ .$$
\end{corollary}

\begin{proof}
First one uses the splitting principle to obtain full flags on $E$ and $F$. Then in view of remark \ref{remark locus} one can reduce to the setting of the theorem by setting $w=\nu_{\textbf{t}}$, where $\textbf{t}=(e,f,r)$. Finally one recalls (\ref{def symmetric}), the definition of the polynomials $\mathfrak{D}^{(\beta)}_{\textbf{t}}$, and observes that the expression obtained only depends on the elementary symmetric functions in the Chern roots (i.e. on the Chern classes) and as a consequence the formula does not depend on the choices made.
\end{proof}

\begin{corollary} 
Theorem \ref{thm connSchubert} and corollary \ref{cor connThom} recover both statements of theorem \ref{thm Schubert} and corollary \ref{cor Thom}.
\end{corollary}
\begin{proof}
It suffices to apply to the equalities the canonical natural transformations $CK^*\rightarrow CH^* $ and $CK^*\rightarrow K_0[\beta,\beta^{-1}]$ arising from the universality of connective $K$-theory. This immediately recovers the results for the Chow ring, while for the Grothendieck ring it is still necessary to set $\beta$ equal to~1.   
\end{proof}

\section{Appendix: Double $\beta$-polynomials}

In \cite{GrothendieckFomin} Fomin and Kirillov propose an alternative definition of $\beta$-polynomials, viewed as coefficients of a particular element $H(\textbf{x})$ of a Hecke algebra $\mathcal{A}_n^{(\beta)}$. As they point out, by following the approach used in \cite{SchubertFomin} for Schubert polynomials, it is also possible to define double $\beta$-polynomials.  This appendix serves the purpose of making explicit the equivalence of the two different definitions and to derive some properties arising from this equivalence. For a more detailed exposition we refer the reader to \cite{DoubleGrothendieckFomin}, with the warning that in this source Grothendieck polynomials are given an extra parameter and therefore coincide with $\beta$-polynomials. We will follow the notations of \cite{GrothendieckFomin} and write $x\oplus y$ for $x+y+\beta xy$ and $\ominus x$ for $-\frac{x}{1+\beta x}$. 

In our context $\mathcal{A}_n^{(\beta)}$ will denote an associative algebra over $\ZZ[\beta][\textbf{x},\textbf{y}]$ generated by elements $u_i\,$, $i\in\{1,\tred,n-1\}$ which satisfy the following three sets of  relations:
\begin{align}
i)\ \  u_iu_j=u_ju_i \quad |i-j|<2\ ;\qquad ii) \ \ u_iu_{i+1}u_i=u_{i+1}u_iu_{i+1}\ ;\qquad iii)\ \  u_i^2=\beta u_i\ .\label{eq relations}
\end{align}
An immediate consequence of these relations is that it is possible to express any product of the generators as a reduced decomposition of an element of $S_n$ and hence every element $\alpha$ can be expressed as $\sum_{\omega\in S_n} A_\omega\cdot \omega$ with $A_\omega\in\ZZ[\beta][\textbf{x},\textbf{y}].$ 

 Given $h_i(x):=1+x\cdot u_i$ and $\alpha_i(x):=h_{n-1}(x)\trecd h_{i}(x)$, we set
$$H(\textbf{x}):=\alpha_1(x_1)\trecd\alpha_{n-1}(x_{n-1})\ .$$
In a similar fashion one defines the analogues for the second set of variables:
 $$\widetilde{\alpha}_i(y):=h_i(y)\trecd h_{n-1}(y)\qquad ;\qquad\widetilde{H}(\textbf{y}):=\alpha_{n-1}(y_{n-1})\trecd\alpha_1(y_1)\ .
$$

\begin{definition}
$$H(\textbf{x},\textbf{y}):=\widetilde{H}(\textbf{y})H(\textbf{x})=\sum_{\omega\in S_n}H_\omega(\textbf{x},\textbf{y}) \cdot\omega\ .$$
\end{definition}

In order to manipulate $H(\textbf{x},\textbf{y})$ into a more convenient expression, one needs two lemmas, whose proof is based on the following equalities:
\begin{align}
i)\ \  h_i(x)h_j(y)=h_j(y)h_i(x) \quad |i-j|<2\quad &;
\quad ii)\ \ h_i(x)h_i(y)=h_i(x\oplus y)\ ; \label{eq h1}
\\\quad iii) \ \ h_i(x)h_{i+1}(x\oplus y)h_i(y)=&h_{i+1}(y)h_i(x\oplus y)h_{i+1}(x)\ . \label{eq h2}
\end{align}
Note that all these formulas can be easily proven through direct calculation by making use of (\ref{eq relations}).

\begin{lemma}
For any variables $x$ and $y$ we have: $\alpha_i(x)\alpha_i(y)=\alpha_i(y)\alpha_i(x)\ ; \  \alpha_i(x)\widetilde{\alpha}_i(y)=\widetilde{\alpha}_i(y)\alpha_i(x)\, .$

\end{lemma}

\begin{proof}
The proof of the first statement consists of formal manipulations using equalities (\ref{eq h1})-(\ref{eq h2}) and mimics the one for Schubert polynomials. For the details see \cite[Lemma 2.1]{SchubertFomin}. The second equality follows from the first since $\widetilde{\alpha}_i(y)=\alpha_i^{-1}(\ominus y)$.
\end{proof}

\begin{lemma}\label{lem alpha xy}
$\widetilde{\alpha}_{n-1}(y_{n_1})\trecd\widetilde{\alpha_i}(y_i)\alpha_i(x)=h_{n-1}(x\oplus y_{n-1})\trecd h_i(x\oplus y_i)\widetilde{\alpha}_{n-1}(y_{n-2})\trecd\widetilde{\alpha}_{i+1}(y_i)$
\end{lemma}

\begin{proof}
Again the proof does not differ from the case of Schubert polynomials, see \cite[Lemma 4.2]{SchubertFomin}
\end{proof}

A repeated use of lemma $\ref{lem alpha xy}$ yields the following

\begin{proposition}\label{prop alternative}
$$H(\textbf{x},\textbf{y})=\prod_{i=1}^{n-1}\prod_{j={n-i}}^1h_{i+j-1}(x_i\oplus y_j)\ ,$$
where the factors are multiplied from left to right. 
\end{proposition}

We are now able to prove the equivalence of the two definitions. Let us recall that $\phi_i$ represents the $\beta$-divided difference operators on $\ZZ[\beta][\textbf{x},\textbf{y}]$. It is easy to see that they can also be viewed as operators on $\mathcal{A}_n^{(\beta)}$ since the coefficients of each element belong to $\ZZ[\beta][\textbf{x},\textbf{y}]$. The next proposition clarifies the effect of these operators  on  $H(\textbf{x})$ and  $H(\textbf{x},\textbf{y})$ and essentially shows that the polynomials $H_\omega$ satisfy the same recursive relation of $\mathfrak{H}_\omega^{(\beta)}$.

\begin{proposition}\label{prop operator}
$1)\ \phi_i H(\textbf{x})=H(\textbf{x})u_i-\beta H(\textbf{x})\quad ;\quad 2)\ \phi_i H(\textbf{x},\textbf{y})=H(\textbf{x},\textbf{y})\cdot u_i -\beta H(\textbf{x},\textbf{y})\ .$
\end{proposition} 

\begin{proof}
For the first identity see \cite[within the proof of theorem 2.3]{GrothendieckFomin}. For the second it suffices to recall the definition of $H(\textbf{x},\textbf{y})$, to observe that the operator is linear with respect to polynomials symmetric in $x_i$ and $x_{i+1}$ and use the first part. 
\end{proof}

\begin{corollary} 
$H(\textbf{x},\textbf{y})=\sum_{\omega\in S_n}\phi_\omega 
H_{\omega_0}(\textbf{x},\textbf{y}) \cdot\omega_0 \omega$ . 
\end{corollary}

\begin{proof}
We prove by induction on $l=l(\omega)$ that $\phi_\omega 
H_{\omega_0}(\textbf{x},\textbf{y})=H_{\omega_0\omega}(\textbf{x},\textbf{y})$. For $l=0$ the statement is trivial. On the other hand, for any $\omega\neq id$ there exists $i$ such that $\omega=\omega's_i$ with $l(\omega)=l(\omega')+1$. One then applies part 2 of proposition  \ref{prop operator} for such $i$ and considers the coefficient of $\omega_0\omega'$ on both sides of the equation. This yields
$$\phi_iH_{\omega_0\omega'}(\textbf{x},\textbf{y})=
(H_{\omega_0\omega's_i}(\textbf{x},\textbf{y})+\beta H_{\omega_0\omega'}(\textbf{x},\textbf{y}))-\beta H_{\omega_0\omega'}(\textbf{x},\textbf{y})=H_{\omega_0\omega}(\textbf{x},\textbf{y})\ ,$$
where $H_\omega(\textbf{x},\textbf{y})+\beta H_{\omega'}(\textbf{x},\textbf{y})$ represents the coefficient for $H(\textbf{x},\textbf{y})\cdot u_i\,$. The statement then follows since  by inductive hypothesis $\phi_{\omega'}H_{\omega_0}(\textbf{x},\textbf{y})=
H_{\omega_0\omega'}(\textbf{x},\textbf{y})$.
\end{proof}
From proposition \ref{prop alternative} one can instantly check that $H_{\omega_0}(\textbf{x},\textbf{y})=\mathfrak{H}_{\omega_0}(\textbf{x},\textbf{y})$ and therefore it follows that the two definitions are equivalent:

\begin{proposition}
For every $\omega\in S_n$ one has $H_{\omega}(\textbf{x},\textbf{y})=
\mathfrak{H}_{\omega}(\textbf{x},\textbf{y})\ .$

\end{proposition}

Now that we have at hand this alternative definition, it is possible to prove the analogue of some properties of double Schubert polynomials. 

\begin{lemma}\label{lem beta pol}
$ \ \mathfrak{H}_{\omega^{-1}}(\textbf{x},\textbf{y})=\mathfrak{H}_{\omega}(\textbf{y},\textbf{x})\ .$
\end{lemma}

\begin{proof}
One only needs to observe that by making use of the first group of relations in (\ref{eq relations}), it is possible to rearrange the factors of  $H(\textbf{x},\textbf{y})=\widetilde{H}(\textbf{x})H(\textbf{y})$ so that they appear precisely in the reverse order of the ones of $\widetilde{H}(\textbf{y})H(\textbf{x})$.  The result then follows because reversing the order of a product of elementary transposition amounts to taking its inverse. 
\end{proof}

\begin{lemma}     \label{lem symmetry}
If $\omega(i)<\omega(i+1)$, then $\mathfrak{H}_{\omega}(\textbf{x},\textbf{y})$ is symmetrical in $x_i$ and $x_{i+1}$. Also, if $\omega^{-1}(i)<\omega^{-1}(i+1)$, then $\mathfrak{H}_{\omega}(\textbf{x},\textbf{y})$ is symmetrical in $y_i$ and $y_{i+1}$.
\end{lemma}

\begin{proof}
For the first statement it suffices to observe that the hypothesis on $\omega$ ensures that $l(\omega)<l(\omega s_i)$ and that by definition $\phi_i(P)$ is symmetric in $x_i$, $x_{i+1}$ for any polynomial $P$. The second statement follows from the first by applying lemma \ref{lem beta pol}.
\end{proof}


\bibliographystyle{siam}
\bibliography{biblio}

 Department of Mathematical Sciences, KAIST, 291 Daehak-ro, Yuseong-gu, Daejeon, 305-701, South Korea
\vspace*{1\baselineskip}

\noindent \textit{E-mail address}: hudson.t@kaist.ac.kr

\end{document}